\documentclass[a4paper, reqno, 11pt]{amsart}

\usepackage[english]{babel}
\usepackage{amsmath}
\usepackage{mathrsfs}
\usepackage{amssymb}
\usepackage{amsthm}
\usepackage{enumerate}
\usepackage{ifthen}
\usepackage{bbm}
\usepackage{color}
\usepackage{xcolor}
\usepackage{graphicx}
\provideboolean{shownotes} 
\setboolean{shownotes}{true}
\usepackage{hyperref}
\usepackage{amsfonts}
\usepackage{listings}
\usepackage{amssymb}
\usepackage{amsbsy}
\usepackage{stmaryrd}
\usepackage{mathtools}
\usepackage{paralist}

\usepackage{float}   
\usepackage{fancybox}		  
\usepackage{verbatim}

\def\XXint#1#2#3{{\setbox0=\hbox{$#1{#2#3}{\int}$ }
\vcenter{\hbox{$#2#3$ }}\kern-.6\wd0}}

\usepackage{geometry}
\newcommand{\margnote}[1]{
\ifthenelse{\boolean{shownotes}}%
{\marginpar{\raggedright\tiny\texttt{#1}}}%
{}%
}

\newcommand{\hole}[1]{
\ifthenelse{\boolean{shownotes}}%
{\begin{center} \fbox{ \rule {.25cm}{0cm}
\rule[-.1cm]{0cm}{.4cm} \parbox{.85\textwidth}{\begin{center}
\texttt{#1}\end{center}} \rule {.25cm}{0cm}}\end{center}}
{}
}

\newcommand{\E}{\mathbb{E}}
\newcommand{\tE}{\tilde{\mathbb{E}}}
\newcommand{\F}{\mathcal{F}}
\newcommand{\tF}{\tilde{\mathcal{F}}}
\newcommand{\tP}{\tilde{\mathbb{P}}}
\newcommand{\bP}{\mathbb{P}}
\newcommand{\Pac}{\mathcal{P}_c^{\mathrm{ac}}(\R^d)}
\newcommand{\R}{\mathbb{R}}

\newcommand{\uh}{\hat{u}}
\newcommand{\mh}{\hat{\mu}}

\newcommand{\e}{\varepsilon}		       
\newcommand{\dive}{\mathop{\mathrm {div}}}

\newcommand{\weaktos}{\stackrel{*}{\rightharpoonup}}
\newcommand{\de}{\,\mathrm{d}}
\renewcommand{\P}{\mathcal{P}(\R^d)}
\newcommand{\Pp}{(\mathcal{P})}
\newcommand{\ppe}{(\mathcal{P}_\e)}
\newcommand{\Pn}{\mathcal{P}_2(\mathbb{R}^d)}

\newcommand{\bproof}{\begin{proof}}
\newcommand{\eproof}{\end{proof}}
\newtheorem{Theorem}{\bf Theorem}
\newtheorem{lemma}{\bf Lemma}

\newtheorem{definition}[lemma]{\bf Definition}
\newtheorem{proposition}[lemma]{\bf Proposition}
\newtheorem{remark}[lemma]{\bf Remark}

\newcommand{\auth}[1]{{ #1}}
\newcommand{\tit}[1]{{\rm #1}}

\newcommand{\jou}[1]{{\it #1}}

\newcommand{\pp}[1]{pp.~#1}

\newcommand{\eps}{\varepsilon}

\newcommand{\schema}[1]{\b{\sc #1}}
\newenvironment{pschema}[1]{\vspace{3mm}
\noindent\begin{Sbox}\begin{minipage}{.95\columnwidth}\vspace{2mm}\begin{center}{\large \schema{#1}}\vspace{5mm}\\
\begin{minipage}{0.9\textwidth}}{\end{minipage}\end{center}\vspace{2mm}\end{minipage}\end{Sbox}\fbox{\TheSbox}\vspace{3mm}}

\renewcommand{\b}[1]{{\bf #1}}

\newcommand{\Lip}{\text{Lip}}

\numberwithin{equation}{section}

\begin{document}

\title[Vanishing viscosity for linear-quadratic mean-field  control problems]{Vanishing viscosity for linear-quadratic mean-field  control problems}

\author[G. Ciampa]{Gennaro Ciampa}
\address[G. Ciampa]{Dipartimento di Matematica ``Tullio Levi Civita''\\ Universit\`a degli Studi di Padova\\Via Trieste 63 \\35131 Padova \\ Italy}
\email[]{\href{ciampa@}{ciampa@math.unipd.it}, \href{gennaro.ciampa@}{gennaro.ciampa@unipd.it}}

\author[F. Rossi]{Francesco Rossi}
\address[F. Rossi]{Dipartimento di Matematica ``Tullio Levi Civita''\\ Universit\`a degli Studi di Padova\\Via Trieste 63 \\35131 Padova \\ Italy}
\email[]{\href{rossi@}{francesco.rossi@math.unipd.it}}

\maketitle

\begin{abstract} 
We consider a mean-field control problem with linear dynamics and quadratic control. We apply the vanishing viscosity method: we add a (regularizing) heat diffusion with a small viscosity coefficient and let such coefficient go to zero. The main result is that, in this case, the limit optimal control is exactly the optimal control of the original problem.
\end{abstract}

\section{Introduction}
Linear-Quadratic problems (LQ from now on) for finite-dimensional systems are the easiest non-trivial examples in optimal control, see \cite{sontag}. Their use in control theory is ubiquitous, in particular as the simplest stabilizers around a nominal trajectory. Thus, any optimal control theory for a new class of systems should confront itself with LQ problems. This article aims to define and solve LQ problems for deterministic mean-field control systems.

Mean-field equations are the natural limit of a large number $N$ of interacting particles when $N$ tends to infinity. The {\bf state} of the system is then a density or, more in general, a measure. We can apply a control, that is an external vector field, to steer the system to a desired configuration or, as in our case, to optimize some cost. The resulting dynamics is called a {\bf mean-field control problem}. See a general treatement of these problems in \cite{BFY,BSYY}.

Mean-field control problems are intimately related to mean-field games, as clearly explained in \cite{BFY}. We recall that mean-field games (first introduced in \cite{HMC1,LL}) describe the limit of a large number of interacting particles in which each particle optimizes a personal cost, in the spirit of Nash differential games. Instead, in mean-field control an (external) controller aims to minimize a global cost for the whole population. In this sense, our contribution aims to provide one more theoretical tool (the vanishing viscosity method) to the study of mean-field control problems. We will focus on LQ models: for mean-field games, they were first studied in \cite{HCM2}, where they are called LQG systems (G stands for Gaussian noise).\\

In this article, we study two deeply connected mean-field control problems. The first corresponds to the mean-field of a deterministic ordinary differential equation, that is a continuity equation for the measure. Its expression is
\begin{equation}\label{e:d1}
\partial_t \mu_t+\dive(b(t,x,\mu_t,u)\mu_t)=0,
\end{equation}
where $\mu_t$ is a time-dependent measure and $b(t,x,\mu_t,u)$ is a vector field, depending on time, space, the measure itself, and the control $u$. The second mean-field control equation is instead the mean-field of a stochastic ordinary differential equation with additive Brownian motion $\sqrt{2\e}W_t$, that is an advection-diffusion equation. Its explicit expression is very similar, as it is 
\begin{equation}\label{e:p1}
\partial_t \mu_t+\dive(b(t,x,\mu_t,u)\mu_t)=\e \Delta \mu_t,
\end{equation}
where $\e \Delta \mu_t$ represents the heat diffusion. It is natural to ask if solutions of \eqref{e:p1} converge to \eqref{e:d1} when $\e\to 0$, i.e. when passing from the probabilistic to the deterministic mean-field problem. Such limit is known as the vanishing viscosity method. It is not hard to prove that we have convergence for a given fixed control $u$, see Lemma \ref{lem:2} below.

We now couple the (deterministic or probabilistic) dynamics for the measure $\mu_t$ with a given running+final cost  \begin{equation}\label{def:J}
   J(\mu,u)= \int_0^T\hspace{-0.2cm}\int_{\R^d} f(t,x,\mu_t,u)\de\mu_t \de t+\int_{\R^d}g(x,\mu_T)\de\mu_T.
\end{equation}
We then have two optimal controls: on one side, the minimizer $u$ of the deterministic optimal control problem coupling \eqref{e:d1} and \eqref{def:J}; on the other side, the family of minimizers $u^\e$ of the probabilistic optimal controls coupling \eqref{e:p1} and \eqref{def:J}, indexed by $\e>0$. In this case, one can again ask the following questions related to vanishing viscosity:
 \begin{itemize}
     \item Do we have convergence of optimal controls $u^\e\to u$?
     \item Do we have convergence of optimal trajectories $\mu^\e_t\to \mu_t$?
     \item Do we have convergence of costs $J(\mu^\e,u^\e)\to J(\mu,u)$?
 \end{itemize}
 
Such questions do not have a general answer. Our main result (Theorem \ref{t:main}) states that, if the dynamics is linear and the cost is quadratic, all answers are positive.\\

The structure of the article is the following. We first fix notation in Section \ref{s-notation}. In Section \ref{s-problem} we define the deterministic and probabilistic problems $\Pp, \ppe$ and state our main result.  We then study the probabilistic problem $\ppe$ in Section \ref{s-ppe}. In Section \ref{s-proof} we prove the main result, that is convergence of the solution of $\ppe$ to the one of $\Pp$. We draw some conclusions in Section \ref{s-conclusions}.

\subsection{Notation}
\label{s-notation}
We denote with $x'$, ($M'$) the transpose of the vector $x$ (matrix $M$). We write that a matrix $Q$ satisfies $Q> 0$, $Q\geq 0$ when it is positive definite (resp. semi-definite).

We will work on the Euclidean space $\R^d$ and we denote by $\P$ the space of probability measures on $\R^d$. The set $\Pn$ will be the subset of the probability measures with finite second moment, that is of measures $\mu$ satisfying
\begin{equation*}
\int_{\R^d}|x|^2\de\mu<\infty.
\end{equation*}

We will mostly focus on the subspace of absolutely continuous measures with compact support $\Pac$. We denote with $\bar{\mu}$ the {\em barycenter} of the measure $\mu$, that is defined as
\begin{equation*}
\bar{\mu}:=\int_{\R^d}\xi\, \mu(\de \xi).
\end{equation*}

We also denote with $\mu_n\weaktos\mu$ the standard weak-$*$ convergence of measures. We recall that it means the following:
\begin{equation*}
\forall \varphi\in C^\infty_c(\R^d) \mbox{~~it holds~~}\int_{\R^d}\varphi\,d\mu_n\to \int_{\R^d}\varphi\,d\mu.
\end{equation*}

Finally, we denote the space of Lipschitz functions as 
$$
\Lip(\R^d):=\left\{f:\R^d\to\R^d\ |\ \exists L>0 ~\mbox{s.t.}~\forall x,y\in \R^d~~|f(x)-f(y)|\leq L |x-y|\right\}.
$$

\section{Problem statement} \label{s-problem}

In this article, we consider two LQ mean-field control problems. This means that:
\begin{itemize}
    \item the dynamics is linear, i.e. the vector field is of the form 
    \begin{equation}\label{def:b}
b(t,x,\mu,u)=A(t)x+B(t)u+\bar{A}(t)\bar{\mu}_t,
\end{equation}
\item the cost is quadratic, i.e the cost $J(\mu,u)$ defined in \eqref{def:J} is of the form
\begin{equation}\label{def:f}
f(t,x,\mu,u)=\frac{1}{2}[ x'Q(t)x+u'R(t)u+(x-S(t)\bar{\mu}_t)'\bar{Q}(t)(x-S(t)\bar{\mu}_t)],
\end{equation}
\begin{equation}\label{def:g}
g(x,\mu)=\frac{1}{2}\left[ x'Q_Tx+(x-S_T\bar{\mu}_T)'\bar{Q}_T(x-S_T\bar{\mu}_T)\right],
\end{equation}
\end{itemize}

Here, all operators $A(t),B(t),\bar{A}(t),Q(t),R(t),S(t)$, $\bar{Q}(t),Q_T,S_T,\bar{Q}_T$ are linear operators $\R^d\to \R^d$, i.e. square matrices. The first 7 of them are continuous functions of time, defined for all $t\in[0,T]$. We will often omit the dependence of the matrices on time, for simplicity of notation.

\begin{remark}
It is interesting to observe that both the dynamics $b$ and the cost $J(\mu,u)$  contains terms depending on the barycenter $\bar\mu_t$ of the state. Such choice is given by the following fact, see \cite{BFY}: in mean-field control problems, it is interesting to steer the state measure either towards its barycenter or far from it (i.e concentrate or dissipate the population).
\end{remark}

We will assume standard symmetry and positive-definiteness of matrices from now on. We also need to add two conditions related to the barycenter terms. They are summarized here:

\newcommand{\M}{{(}M{)}}
\begin{pschema}{{(}M{)}}
\begin{itemize}
\item {\bf Running cost:} for all $t\in[0,T]$, the matrices $Q(t),\bar{Q}(t), R(t)$ are symmetric with $Q(t),\bar{Q}(t)\geq 0$ and $R(t)>0$.

\item {\bf Barycenter in the running cost:} for all $t\in[0,T]$ it holds
\begin{equation}\label{cond:qs}
Q(t)+(I-S(t))'\bar{Q}(t)(I-S(t))\geq 0.
\end{equation}

\item {\bf Final cost:} it holds $Q_T,\bar{Q}_T$ symmetric, both satisfying $Q_T,\bar{Q}_T\geq 0$.

\item {\bf Barycenter in the final cost:} it holds 
\begin{equation}\label{cond:qsT}
Q_T+(I-S_T)'\bar{Q_T}(I-S_T)\geq 0.
\end{equation}
\end{itemize}

\end{pschema}

Observe that matrices \eqref{cond:qs}-\eqref{cond:qsT} are symmetric by definition.\\

We are now ready to define the deterministic problem $\Pp$:

\begin{pschema}{Problem $\Pp$}
Find
\begin{equation*}
    \min_{u\in\mathcal{U}}J(\mu_t,u),
\end{equation*}
such that $\mu_t\in C([0,T];\Pn)$ is a solution of 
\begin{equation}\label{e:dP}
\begin{cases}
\partial_t \mu_t+\dive(b(t,x,\mu_t,u)\mu_t)=0,\\
\mu_t|_{t=0}=\mu_0.
\end{cases}
\end{equation}
The cost $J(\mu,u)$ is quadratic, given by \eqref{def:J}, \eqref{def:f}, \eqref{def:g} and the matrices satisfy \M.
The dynamics is linear, given by \eqref{def:b}. The initial state satisfies $\mu_0\in\Pac$.

The set of admissible controls is
$$
\mathcal{U}:=L^1((0,T);\Lip(\R^d)).
$$
\end{pschema}

The probabilistic (or viscous) problem is very similar:

\begin{pschema}{Problem $\ppe$}
Take $\Pp$ and replace the dynamics \eqref{e:dP} with

\begin{equation}\label{e:ppe}
\begin{cases}
\partial_t \mu_t+\dive(b(t,x,\mu_t,u)\mu_t)=\e \Delta \mu_t,\\
\mu_t|_{t=0}=\mu_0.
\end{cases}
\end{equation}
\end{pschema}

As already stated above, the crucial difference between $\Pp$ and $\ppe$ is given by the presence of a viscosity term $\e \Delta\mu_t$. This corresponds to the fact that the underlying dynamics in $\Pp$ is deterministic, while in $\ppe$ it is associated to an additive Brownian motion $\sqrt{2\eps}W_t$. From the mathematical point of view, the dynamics of $\ppe$ is easier to study than $\Pp$, as the viscosity term (that is a heat diffusion) ensures stronger regularity of the solution. For this reason, it is classical to study the regular viscous problem $\ppe$ and hope to pass to the limit to infer something on the original deterministic problem $\Pp$. Such method, known as vanishing viscosity, can be traced back to \cite{K1, K2, Mc}.

Such more theoretical aspect also has an applied interest: indeed, regularity of the viscous problem implies that standard numerical methods to solve it can be applied, then the deterministic problem can be solved by convergence.

Our aim is then to apply the vanishing viscosity method to our problem: we study the convergence of solutions of $\ppe$ to solutions of $\Pp$ for $\e\to 0$. In the LQ setting, we have convergence, as stated in our main result.
\begin{Theorem}\label{t:main}
Let $(\mh^\e,\uh^\e)$ be solutions of $\ppe$. Then, there exists a solution $(\mh,\uh)$ of $\Pp$ such that, for $\e\to 0$, it holds:
\begin{itemize}
\item[$(i)$] $\uh^\e\to \uh$ in $\mathcal{U}=L^1((0,T);\Lip(\R^d))$;
\item[$(ii)$] $\mh^\e\to \mh$ in $C([0,T],\Pn)$;
\item[$(iii)$] $J(\mh^\e,\uh^\e)\to J(\mh,\uh)$.
\end{itemize}
\end{Theorem}
The proof of the Theorem is given in Section \ref{s-proof}.

\begin{remark} 
One usually looks for admissible controls $u\in L^1((0,T);L^1(\R^d;\de \mu_t))$ in $\Pp$ and $L^1((0,T);L^1(\R^d;\de \mu^\e_t))$ in $\ppe$. This is indeed the minimum requirement in order to give a (weak) meaning to equations \eqref{e:dP} and \eqref{e:ppe}. However, for our purposes it is not restrictive to search for a Lipschitz optimal control for the following reasons:
\begin{itemize}
    \item in $\Pp$, conditions \eqref{cond:qs}-\eqref{cond:qsT} imply that the coercivity condition of \cite{BR} is satisfied; this ensures that there exists at least one Lipschitz optimal control;
    \item we know {\em a priori} that even if in $\ppe$ we look for controls in the larger class $L^1((0,T);L^1(\R^d;\de \mu^\e_t))$, the optimum is Lipschitz. See \cite[Lemma 6.18]{CD} and the construction in Section IV.
\end{itemize}
Therefore, by restricting ourselves to Lipschitz controls, we have the advantage that equation \eqref{e:dP} is well-posed, see Theorem \ref{thm:well_pos_cont} below. Moreover, the functional spaces do not depend on the solutions $\mu_t,\mu^\e_t$  anymore.
\end{remark}

\section{Mean-field equations}

In this section, we recall the main tools that allow to study mean-field equations. The most relevant examples of such equations are the ones in which the vector field is given by a convolution, i.e. of the form
\begin{equation}\label{eq:vlasov}
     \partial_t \mu_t+\dive(V(\mu_t\star H)\,\mu_t)=0.
\end{equation}
Convolution indeed describes long-range interaction between particles. The mean-field limit of several fundamental agent-based models have been described and studied, such as the limits of bounded confidence models \cite{MFJQ} or Cucker-Smale \cite{HT}. More in general, we study the following Cauchy problem:
\begin{equation}\label{eq:cont}
    \begin{cases}
    \partial_t \mu_t+\dive(v[\mu_t](t,\cdot)\,\mu_t)=0,\\
    \mu_t|_{t=0}=\mu_0,
    \end{cases}
\end{equation}
where the vector field $v[\mu]:(0,T)\times\R^d\to\R^d$ and the initial probability measure $\mu_0\in\P$ are given. Observe that the vector field depends on the whole measure $\mu$ itself. It represents the fact that the dynamics in a point of the density actually depends on the density elsewhere, due to long-range interactions. The presence of such phenomenon, known as non-locality, requires to develop a specific theory to ensure well-posedness of \eqref{eq:cont}.

\subsection{Well-posedness of the deterministic Cauchy problem}

In this section, we study well-posedness of \eqref{eq:cont}. Such theory is based on Wasserstein distance, that we recall here.
\begin{definition}
Let $\mu,\nu\in \Pn$. We say that $\gamma\in\mathcal{P}(R^{2d})$ is a transport plan between $\mu$ and $\nu$ – denoted by $\gamma \in \Pi(\mu,\nu)$ – provided that $\gamma(A \times \R^d) = \mu(A)$ and $\gamma(\R^d \times B) = \nu(B)$ for any pair of Borel sets $A,B \subset \R^d$. Given two measures $\mu,\nu\in \Pn$, the Wasserstein distance between $\mu$ and $\nu$ is given by
\begin{equation*}
W_2^2(\mu,\nu)=\min_{\gamma\in \Pi(\mu,\nu)}\left\{ \int_{\R^d\times\R^d} |x-y|^2\,\gamma(\de x,\de y) \right\}.
\end{equation*}
\end{definition}
\begin{remark}
The choice of the Wasserstein distance $W_2$ is related to the fact that we deal with an optimal control problem with quadratic cost.
\end{remark}

We will consider the space $\Pn$ endowed with the  Wasserstein distance $W_2$ from now on. It is remarkable to observe that the Wasserstein distance metrizes the weak-$*$ topology of probability measures. More precisely, it holds:
\begin{equation}\label{conv_w}
W_2(\mu,\mu_n)\to 0 \Longleftrightarrow \begin{cases}
\hspace{1.4cm}\mu_n\weaktos\mu,\\
\displaystyle\int_{\R^d}|x|^2\mu_n(\de x)\to \int_{\R^d}|x|^2\mu(\de x).
\end{cases}
\end{equation}
We refer to \cite{villani} for an overview of Wasserstein spaces. 

We are now ready to state the precise meaning of weak solution for \eqref{eq:cont} and the corresponding result of existence and uniqueness, proved in \cite{PR}.
\begin{definition}
A weak solution of \eqref{eq:cont} is a probability measure $\mu_t\in C([0,T];\P)$ such that
\begin{equation*}
    \frac{\de}{\de t}\int_{\R^d} \varphi(x)\mu_t(\de x)=\int_{\R^d}v[\mu_t](t,x)\cdot\nabla\varphi(x)\mu_t(\de x),
\end{equation*}
for all test functions $\varphi\in C^\infty_c(\R^d)$.
\end{definition}

\begin{Theorem}\label{thm:well_pos_cont}
Let $v$ be uniformly Lipschitz with respect to the Wasserstein distance on $\Pn$ and the Euclidean distance in $\R^d$, i.e. there exists a constant $L$ such that
\begin{equation}\label{def:lip}
    |v[\mu](t,x)-v[\nu](t,y)|\leq L\left(W_2(\mu,\nu)+|x-y|\right),
\end{equation}
for all $t\in\R$, $\mu,\nu\in\Pn$, and $x,y\in\R^d$.
Then, for each $\mu_0\in\Pac$ there exists a unique solution to \eqref{eq:cont}.
\end{Theorem}

The key consequence of this theorem for the study of $\Pp$ is that the exponential map, associating to each Lipschitz control $u$ the corresponding solution of \eqref{e:dP}, is well defined.

\subsection{Derivative with respect to a measure}

In this section, we choose the correct concept of derivative of a functional with respect to a measure. Since we deal with minimization of a functional $J$, it is useful to have a first-order condition with respect to perturbations of the state $\mu$.

We recall that there are several different concepts of derivatives with respect to measures, see e.g. \cite{CD}. For our problem, we need the so-called \emph{L-derivative}.
Let $(\Omega,\mathcal{F},\mathbb{P})$ be an atomless probability space. Given a map $h:\Pn\to\R$ we define the {\em lifting}
$$
\tilde{h}(X)=h(\mathcal{L}(X)),\hspace{0.3cm}\forall X\in L^2(\Omega;\R^d),
$$
where $\mathcal{L}(X):=X\#\mathbb{P}$. Note that $\mathcal{L}(X)\in\Pn$, since $X\in L^2(\Omega;\R^d)$.
\begin{definition}
A function $h:\Pn\to\R$ is said to be {\em L-differentiable} at $\mu_0\in\Pn$ if there exists a random variable $X_0$ with law $\mu_0$ such that the lifted function $\tilde{h}$ is Fr\'echet differentiable at $X_0$.
\end{definition}
The Fr\'echet derivative of $\tilde{h}$ at $X$ can be viewed as an element of $L^2(\Omega;\R^d)$ and we can denote it by $D\tilde{h}(X)$. It is important that the differentiability of $h$ does not depend upon the particular choice of $X$.
\begin{proposition}
Let $h:\Pn\to\R$ and $\tilde{h}$ its extension. Let $X,X'\in L^2(\Omega;\R^d)$ with the same law $\mu$. If $\tilde{h}$ is Fr\'echet differentiable at $X$, then $\tilde{h}$ is Fr\'echet differentiable at $X'$ and $(X,D\tilde{h}(X))$ has the same law as $(X',D\tilde{h}(X'))$.
\end{proposition}
\begin{proposition}
Let $h:\Pn\to\R^d$ be L-differentiable than for any $\mu_0\in\Pn$ there exists a measurable function $\xi:\R^d\to\R^d$ such that for all $X\in L^2(\Omega;\R^d)$ with law $\mu_0$, it holds that $D\tilde{h}(X)=\xi(X)$ $\mu_0$-almost surely.
\end{proposition}
With this notations, the equivalence class of $\xi\in L^2(\R^d,\mu_0;\R^d)$ is uniquely defined and we denote it by $\partial_\mu h(\mu_0)$. We call {\em L-derivative} of $h$ at $\mu_0$ the function
$$
\partial_\mu h(\mu_0)(\cdot):x\in\R^d\mapsto\partial_\mu h(\mu_0)(x).
$$
Finally, a function $h:\Pn\to\R^d$ is said to be {\em L-convex} if it is L-differentiable and
$$
h(\mu)-h(\mu')-\E[\partial_\mu h(\mu_0)(X)\cdot(X-X')]\geq 0,
$$
whenever $X,X'\in L^2(\Omega;\R^d)$ with law, respectively, $\mu,\mu'$.

\section{The viscous LQ problem $\ppe$} \label{s-ppe}

In this section, we solve the viscous problem $\ppe$. We prove existence and uniqueness of the optimal control and provide an explicit expression by a cascade of two Riccati equations. This section is mostly based on \cite{B,BFY,CD}.

The main idea to solve $\ppe$ is to use stochastic control to find the optimal pair $(\mu^\e_t,u^\e)$ for it. Indeed, we may think of $\mu^\e$ as the law of a stochastic process $X^\e$ which solves the stochastic differential equation 
\begin{equation}\label{eq:SDE}
\begin{cases}
\de X^\e_t=b(t,X^\e_t,\mu^\e_t,\alpha)\de t+\sqrt{2\e}\de W_t,\\
X^\e_t|_{t=0}=X_0.
\end{cases}
\end{equation}
Here $W_t$ is a standard Brownian motion, $\alpha$ is the control, $X_0$ is a random variable independent from $W_t$ with law $\mu_0$, and the equation has to be understood in the It\^o sense. With these notations the cost functional can be rewritten as
\begin{equation*}
    J(\alpha)=\E\left[ \int_0^T f(t,X^\e_t,\mu^\e_t,\alpha_t)\de t +g(X^\e_T,\mu^\e_T) \right],
\end{equation*}
where the control $\alpha_t$ is a measurable process with values in $\R^d$. It moreover satisfies $$
\E\left[ \int_0^T|\alpha_t|^2\de t\right]<\infty.
$$

When the data of the optimization problem are smooth enough, solving $\ppe$ and applying the stochastic approach  just described are equivalent. This is the case of the present LQ setting; we refer to \cite{CD,LP} for the precise description of the two approaches and the connection between them. In what follows, we solve the LQ mean-field optimal control problem by using the stochastic approach and then we recover the solution of $\ppe$  to prove the vanishing viscosity method.

\subsection{The adjoint variable and stochastic maximum principle}

In this section we define the adjoint process of a controlled state $X_t$ and we recall the Stochastic Pontryagin Maximum Principle (SPMP) for optimality. Since we will give general definitions, in this subsection we drop the superscript $\e$ for simplicity. We will use the notation $(\tilde{\Omega},\tF,\tP)$ for a copy of $(\Omega,\F,\bP)$ and $\tE$ for the expectation under $\tP$.

First, we define the Hamiltonian $H$ as
\begin{equation}\label{def:H}
H(t,x,\mu,y,\alpha)=b(t,x,\mu,\alpha)\cdot y+f(t,x,\mu,\alpha),
\end{equation}
for $(t,x,\mu,y,\alpha)\in [0,T]\times\R^d\times\Pn\times\R^d\times\R^d$.
\begin{definition}
We call an {\em adjoint processes} of $X^\e_t$ any couple $(Y_t,Z_t)$ satisfying the equation
\begin{equation}\label{eq:adj_st}
    \begin{cases}
    -\de Y_t=(\nabla H(t,X_t,\mu_t,Y_t,Z_t,\alpha_t)+\tilde{\E}[\partial_\mu H(t,\tilde{X}_t,\mu_t,\tilde{Y}_t,\tilde{Z}_t,\tilde{\alpha}_t)(X_t)])\de t+Z_t\de W_t,\\
    Y_T=\nabla g(X_T,\mu_T)+ \tilde{\E}[\partial_\mu g(\tilde{X}_T,\mu_T)(X_T)],
    \end{cases}
\end{equation}
where $(\tilde{X}_t,\tilde{Y}_t,\tilde{Z}_t,\tilde{\alpha}_t)$ is an independent copy of $(X_t,Y_t,Z_t,\alpha_t)$ defined on the space $(\tilde{\Omega},\tilde{\mathcal{F}},\tilde{\mathbb{P}})$.
\end{definition}
We have the following sufficient condition for optimality associated to the SPMP, see \cite{CD}.
\begin{Theorem}\label{teo:ott-st}
Let $b,f,g$ be Linear-Quadratic, i.e. \eqref{def:b}-\eqref{def:f}-\eqref{def:g} hold. Let $\alpha_t$ be an admissible control, $X_t$ the corresponding controlled state process,
and $(Y_t,Z_t)$ the corresponding adjoint processes. Assume that:
\begin{itemize}
    \item $(x,\mu)\mapsto g(x,\mu)$ is convex,
    \item $(x,\mu,\alpha)\mapsto H(t,x,\mu,Y_t,Z_t,\alpha)$ is convex.
\end{itemize}
If it holds $\mathcal{L}^1\otimes\mathbb{P}$-a.e. that
\begin{equation*}
    H(t,X_t,\mu_t,Y_t,Z_t,\alpha^*)=\inf_{\alpha} H(t,X_t,\mu_t,Y_t,Z_t,\alpha),
\end{equation*}
then $\alpha^*$ is an optimal control. Moreover, if $H$ is strictly convex in $\alpha$, then the optimal control is also unique.
\end{Theorem}


The convexity in the measure variable in the Theorem above is intended as L-convexity. By assuming $H$ and $g$ defined as \eqref{def:H} and \eqref{def:g}, the hypothesis of Theorem \ref{teo:ott-st} are satisfied and we can compute explicitly the unique optimal control. This is the content of the next section.

\subsection{Computation of the optimal control}

In this section, we explicitly find the optimal control $\hat{u}^\e$ of $\ppe$, showing moreover that it does not depend on the parameter $\e$. With this goal, we first solve the forward-backward system of equations given by the SPMP. Then we reconstruct the optimal control of $\ppe$.

By Theorem \ref{teo:ott-st}, the optimal control is given by the minimizer (in the control variable) of the Hamiltonian, i.e.
\begin{equation*}
    \hat{\alpha}=\hat{\alpha}(t,x,\mu,y)=-R^{-1}B'y,
\end{equation*}
while the L-derivative of $H$ is
\begin{equation*}
    \partial_\mu H(t,x,\mu,y)(x')=\bar{A}'y-S'\bar{Q}(x-S\bar{\mu}).
\end{equation*}
Then, by plugging $\hat{\alpha}$ in \eqref{eq:SDE} and \eqref{eq:adj_st}, we obtain the following forward-backward system:
\begin{equation}\label{eq:FB-PMP}
    \begin{cases}
    \de X^\e_t= \left(AX^\e_t-BR^{-1}B'Y^\e_t+\bar{A}\E[X^\e_t]\right)\de t+\sqrt{2\e}\de W_t,\\
    X^\e_t|_{t=0}= X_0,\\
    -\de Y^\e_t=\left( A'Y^\e_t+(Q+\bar{Q})X^\e_t-\bar{Q}S\E[X^\e_t]\right)\de t+\left(\bar{A}'\E[Y^\e_t]-S'\bar{Q}(I-S)\E[X^\e_t]\right)\de t +Z^\e_t\de W_t,\\
    Y^\e_T= (Q+\bar{Q})X^\e_T+(S'_T\bar{Q}_TS_T-S'_T\bar{Q}_T-\bar{Q}_TS_T)\E[X^\e_T].
    \end{cases}
\end{equation}
By taking the expectations in \eqref{eq:FB-PMP}, we have that $\bar{x}^\e_t:=\E[X^\e_t],\bar{y}^\e_t:=\E[Y^\e_t]$ satisfies the system
\begin{equation}\label{eq:FB-E}
    \begin{cases}
    \dot{\bar{x}}^\e_t= (A+\bar{A})\bar{x}^\e_t-BR^{-1}B'\bar{y}^\e_t,\\
    \bar{x}^\e_0= \bar{\mu}_0,\\
    -\dot{\bar{y}}^\e_t= (Q+(I-S)'\bar{Q}(I-S))\bar{x}^\e_t + (A+\bar{A}')\bar{y}^\e_t,\\
    \bar{y}^\e_T=(Q_T+(I-S_T)'\bar{Q}_T(I-S_T))\bar{x}^\e_T.\\
    \end{cases}
\end{equation}
The system \eqref{eq:FB-E} can be associated to the following finite-dimensional optimal control problem
$$
\begin{aligned}
    \min_{w}&\left( \frac{1}{2}\int_0^T \left(\chi'M\chi+w'Rw \right) \de t+\frac{1}{2}\chi(T)'M_T\chi(T)\right),\\
&\begin{cases}
\dot{\chi}(t)=(A+\bar{A})\chi(t)+Bw(t),\\
\chi(0)=\bar{\mu}_0,
\end{cases}
\end{aligned}
$$
where $M=Q+(I-S)'\bar{Q}(I-S)$ and $M_T=Q_T+(I-S_T)'\bar{Q_T}(I-S_T)$. Because of the assumptions \eqref{cond:qs}, \eqref{cond:qsT}, the problem above is a finite dimensional LQ control problem, where $\bar{x}^\e_t$ is the optimal trajectory and $\bar{y}_t^\e$ the associated co-state. To solve this problem, it is enough to find the unique solution $\Sigma$ of the Riccati equation
\begin{equation*}
\begin{cases}
\dot{\Sigma}+\Sigma (A+\bar{A})+(A+\bar{A})' \Sigma-\Sigma BR^{-1}B'\Sigma+Q+(I-S)'\bar{Q}(I-S)=0,\\
\Sigma(T)=Q_T+(I-S_T)'\bar{Q}_T(I-S_T),
\end{cases}
\end{equation*}
which is symmetric positive definite. It then holds
\begin{equation}\label{def:EY}
    \bar{y}^\e_t=\Sigma\bar{x}^\e_t,
\end{equation}
and substituting in \eqref{eq:FB-E} we get 
\begin{equation}\label{eq:EX}
    \begin{cases}
    \dot{\bar{x}}^\e_t= (A+\bar{A}-BR^{-1}B'\Sigma)\bar{x}^\e_t,\\
    \bar{x}^\e_0= \bar{\mu}_0.
    \end{cases}
\end{equation}
Since \eqref{eq:EX} is a linear ODE with smooth coefficients, it admits a unique solution $\bar{x}_t$, which therefore does not depend on $\e$. The same holds for $\bar{y}_t=\bar{y}^\e_t$ as a consequence of \eqref{def:EY}. Then, once $\E[X^\e_t]$ and $\E[Y^\e_t]$ are replaced with $\bar x_t, \bar y_t$, the system \eqref{eq:FB-PMP} can be associated to a standard stochastic control problem with strictly convex Hamiltonian, hence it admits a unique solution, see \cite{CD}. The affine structure of \eqref{eq:FB-PMP} suggests to write $Y^\e_t=P^\e X^\e_t + p^\e_t$, where
\begin{equation}\label{eq:Pp}
    \begin{cases}
    \dot{P}^\e+A'P^\e+P^\e A-P^\e BR^{-1}B'P^\e+Q+\bar{Q}=0,\\
    P^\e(T)=Q_T+\bar{Q}_T,\\
    \dot{p}^\e_t+(A'-P^\e BR^{-1}B')p^\e_t+(\bar{A}'\Sigma+P^\e\bar{A}+S'\bar{Q}S-S'\bar{Q}-\bar{Q}S) \bar{x}_t=0,\\
    p^\e(T)=(S_T'\bar{Q}_TS_T-S_T'\bar{Q}_T-\bar{Q}_TS_T)\bar{x}_t,\\
    Z^\e_t=\sqrt{2\e} P^\e_t.
    \end{cases}
\end{equation}
Arguing as before, the first equation in \eqref{eq:Pp} is of Riccati-type, then it admits a unique solution $P^\e$. Hence, uniqueness of the solution implies that the matrix $P^\e$ does not depend on $\e$. Therefore, once we have the matrix $P$, the equation for $p^\e$ is a linear ODE and again by uniqueness it admits a unique solution, which then does not depend on $\e$. On balance, $Z^\e_t$ is the only term which depends on the viscosity.\\
We can now write the optimal control $\hat{\alpha}$ as
\begin{equation*}
\hat{\alpha}_t=\hat{u}(t,X^\e_t)=-R^{-1}B'PX^\e_t-R^{-1}B'p,
\end{equation*}
which is in feedback form. Finally, by the connection with $\ppe$ described above, we know that the optimal control $\hat{u}$ is then given by
\begin{equation*}
\hat{u}(t,x)=-R^{-1}B'Px-R^{-1}B'p,   
\end{equation*}
which does not depend on $\e$, and the equation solved by the optimal trajectory $\hat{\mu}_t^\e$ is 
\begin{equation}\label{eq:opt_traj_e}
\begin{cases}
\partial_t \hat{\mu}_t^\e+\dive\left(\hat{b}(t,x)\hat{\mu}^\e_t\right)=\e\Delta \hat{\mu}_t^\e,\\
\hat{\mu}_t^\e|_{t=0}=\mu_0,
\end{cases}
\end{equation}
where the vector field $\hat{b}$ is given by
\begin{equation*}
    \hat{b}(t,x)=[A-BR^{-1}B'P]x+\bar{A}\bar{x}_t-BR^{-1}B'p.
\end{equation*}
The equation \eqref{eq:opt_traj_e} is a linear advection-diffusion equation. It is still non-local, since the dependence on the barycenter is contained in the term $\hat{b}$.

\section{Proof of the main result} \label{s-proof}
In this section we prove our main result, that is Theorem \ref{t:main}. We show that the unique optimal pair $(\hat{\mu}^\e_t,\hat{u})$ of $\ppe$ converges to an optimal pair $(\hat{\mu}_t,\hat{u})$ for $\Pp$. 

\subsection{Preliminary lemma} 
In this subsection we prove a stability lemma for the continuity equation; we then apply it to our optimal control problem. It is worth to note that we assume that the control $u$ is fixed in the following Lemma.

\begin{lemma}\label{lem:2}
Let $b:\Pn\to\Lip((0,T)\times\R^d)$ be a vector field satisfying \eqref{def:lip}, such that there exists $C$ for which
$$
|b(t,x,\mu)|\leq C\left(1+|x|+\int_{\R^d}|y|\,\mu(\de y)\right).
$$
Let $\mu_t,\mu^\e_t$ be the unique solution of the deterministic \eqref{e:dP} and the viscous equation \eqref{e:ppe} with vector field $b$, respectively. It then holds
\begin{equation}\label{conv:we}
\lim_{\e\to 0}\sup_{t\in[0,T]}W_2(\mu^\e_t,\mu_t)=0.
\end{equation}
Moreover, the barycenter and costs converge too:
\begin{eqnarray*}
&&\bar{\mu}^\e_t\to\bar{\mu_t},\,\hspace{0.3cm}\mbox{as }\e\to 0,\,\mbox{ uniformly in }[0,T],\\
&&
J(\mu^\e_t,u)\to J(\mu_t,u), \hspace{0.3cm}\mbox{as }\e\to 0.
\end{eqnarray*}
\end{lemma} 
\begin{proof}
The core of the proof is to prove \eqref{conv:we}. Since the total mass is preserved for an advection-diffusion equation, and by standard weak compactness argument (see \cite{villani}) we know that, up to sub-sequences, there exists a measure $\lambda_t\in L^\infty((0,T);\P)$ such that
\begin{equation}\label{weak2}
\mu_t^\e\weaktos \lambda_t \,\mbox{ in } L^\infty((0,T);\mathcal{M}(\R^d)).
\end{equation}
Due to non-linearity in the vector field, a weak convergence as in \eqref{weak2} is not enough to pass to the limit in the equation. However, since $\mu^\e_t,\mu_t\in\Pn$, one can use $|x|^2$ as a test function in the equations. By a Gronwall type argument, thanks to the growth assumptions on $b$, it holds
\begin{equation*}
    \sup_{t\in[0,T]}\int_{\R^d}|x|^2\,\mu^\e_t(\de x)<C,
\end{equation*}
where $C>0$ is a constant independent on $\e$. This implies that any weak limit satisfies $\lambda_t\in \Pn$. Then, by properly choosing the test function in $\eqref{e:ppe}$ and, again by the growth assumptions on $b$, it also holds
\begin{equation*}
    \sup_{t\in[0,T]}\int_{B_R^c}|x|^2\,\mu^\e_t(\de x)\to 0,\hspace{0.3cm}\mbox{uniformly as }R\to\infty.
\end{equation*}
Such a result, together with weak convergence \eqref{weak2}, implies that 
\begin{equation}\label{conv:w2nu}
    \lim_{\e\to 0}\sup_{t\in[0,T]}W_2(\mu^\e_t,\lambda_t)=0.
\end{equation}
Because of the uniform Lipschitz assumption on $b$, the convergence in \eqref{conv:w2nu} is enough to pass to the limit in \eqref{e:ppe} and hence uniqueness for \eqref{e:dP} implies $\mu_t=\lambda_t$. Finally, convergence of the barycenter and of the costs follows from \eqref{conv:we}, by a direct computation.
\end{proof}



\subsection{Proof of Theorem \ref{t:main}}
We are now able to prove our main result Theorem \ref{t:main}.

{\it Proof of Theorem \ref{t:main}}. First of all, by the analysis given in Section IV, the convergence of controls $\uh^\e\to \uh$ is a direct consequence of the fact that the optimal control $\uh$ does not depend on $\e$. Then, we consider $\mh_t$ the unique solution of \eqref{e:dP} with control $\uh$. Since the control does not depend on the viscosity parameter, the convergence of optimal trajectory $\mh^\e\to\mh$ and of the cost $J(\mh^\e,\uh)\to J(\mh,\uh)$ is a direct consequence of Lemma \ref{lem:2}.

To conclude, we must show that $(\mh_t,\uh)$ is actually an optimal pair for $\Pp$. Let $u\neq\uh$ be a Lipschitz control and $\mu_t$ the corresponding trajectory. We define $\mu^\e_t$ to be the unique solution of \eqref{e:ppe} with control $u$ and, since $(\mu_t^\e,u)$ is an admissible pair for $\ppe$ and $(\mh^\e_t,\uh)$ is optimal, it holds
\begin{equation}\label{a}
J(\hat{\mu}^\e_t,\hat{u}) < J(\mu^\e_t,u).
\end{equation}
By $(iii)$ and Lemma \ref{lem:2}, it holds
\begin{equation}\label{b}
J(\hat{\mu}_t,\hat{u})=\lim_{\e\to 0} J(\hat{\mu}^\e_t,\hat{u}), \hspace{0.4cm} J(\mu_t,u)=\lim_{\e\to 0}J(\mu^\e_t,u),
\end{equation}
and combining \eqref{a} and \eqref{b} we get that
\begin{equation*}
J(\hat{\mu}_t,\hat{u})\leq J(\mu_t,u),
\end{equation*}
for any admissible pair $(\mu_t,u)$. Then, $(\mh_t,\uh)$ is an optimal pair for $\Pp$ and the proof is complete. \hfill $\blacksquare$

\section{CONCLUSIONS AND FUTURE PERSPECTIVES} \label{s-conclusions}

In this article, we proved that the vanishing viscosity method works properly when applied to the deterministic LQ mean-field control problem $\Pp$: one can add a small noise of amplitude $\eps$, solve the optimal control problem, and is ensured that for $\eps\to 0$ the limit is the solution of the deterministic LQ mean-field problem.

It is interesting now to study more general mean-field optimal control problems, in particular with coercive costs as in \cite{BR}. Moreover, it will be very interesting to study coupled mean-field controlled systems and mean-field games, i.e. with two levels of optimization.





\section*{ACKNOWLEDGMENT}

This work is funded by the University of Padua under the STARS Grants programme CONNECT: Control of Nonlocal Equations for Crowds and Traffic models.



\begin{thebibliography}{99}

\bibitem{B} 
\auth{M. Bardi}, \tit{Explicit solutions of some linear-quadratic mean field games}, \jou{Networks and Heterogeneous Media}, 7(2), \pp{243--261}, 2012.


\bibitem{BFY} 
A. Bensoussan, J. Frehse, P. Yam, \textit{Mean Field Games and Mean Field Type Control Theory}, Springer, New York, 2013. 


\bibitem{BSYY} 
A. Bensoussan, K.C.J. Sung, S.C.P Yam, S.P. Yung, \textit{Linear-quadratic mean field games}, technical report, 2011.


\bibitem{BR} 
\auth{B. Bonnet, F. Rossi}, \tit{Intrinsic Lipschitz Regularity of Mean-Field Optimal Controls},  \jou{SIAM J Control}, to appear, arXiv:1908.04183.

\bibitem{MFJQ}
\auth{M. Caponigro, B. Piccoli, F. Rossi, E. Trélat}, \tit{Mean-Field Sparse Jurdjevic-Quinn control}, \jou{Math. Models Meth. Appl. Sc.}, Vol. 27, No. 7, \pp{1223--1253}, 2017.



\bibitem{CD} 
R. Carmona, F. Delarue, \textit{Probabilistic Theory of Mean Field Games with Applications I}, Springer, New York, 2018.




\bibitem{HT}
\auth{S.-Y. Ha, E. Tadmor}, \tit{From particle to kinetic and hydrodynamic descriptions of flocking}, \jou{Kinet. Relat. Models} 1, \pp{415--435}, 2008.


\bibitem{HMC1} 
\auth{M. Huang, R.P. Malham\'e, P.E. Caines}, \tit{Large population stochastic dynamic games: closed-loop MCKean-Vlasov systems and the nash certainty equivalence principle}, \jou{ Communications in Information and Systems}, 6(3), \pp{221--252}, 2006.


\bibitem{HCM2} 
\auth{M. Huang, P.E. Caines, R.P. Malham\'e}, \tit{An invariance principle in large population
stochastic dynamic games}, \jou{Journal of Systems Science and Complexity}, 20(2), \pp{162--172}, 2013.


\bibitem{K1}
\auth{S. N. Kru\v{z}kov}, \tit{Generalized solutions of the Hamilton-Jacobi equations of Eikonal type I}, \jou{Math. USSR-Sb} 27, \pp{406--446}, 1975.


\bibitem{K2}
\auth{S. N. Kru\v{z}kov}, \tit{Generalized solutions of nonlinear first order equations in several indipendent variables I}, \jou{Math. Sb} 70(112), \pp{394--415}, 1966. 


\bibitem{LL}
\auth{J.M. Lasry, P.L. Lions}, \tit{Mean field games}, \jou{Japanese journal of mathematics} 2.1, \pp{229--260}, 2007.



\bibitem{LP}
\auth{M. Lauri\`ere, O. Pironneau}, \tit{Dynamic Programming for Mean-Field type Control}, \jou{Compte Rendus de l'acad\'emie des sciences}, Serie I math\'ematiques 352, \pp{707--712}, 2014.


\bibitem{Mc}
\auth{F. J. McGrath}, \tit{Nonstationary Plane Flow of Viscous and Ideal Fluids}, \jou{Arch. Rational Mech. Ana.} 27, \pp{328--348}, 1968.


\bibitem{PR} 
\auth{B. Piccoli, F. Rossi}, \tit{Transport equation with nonlocal velocity in Wasserstein spaces: convergence of numerical schemes}, \jou{Acta Applicandae Mathematicae}, 124, \pp{73--105}, 2013.


\bibitem{sontag} E. D. Sontag, \tit{Mathematical Control Theory}, Springer, 1998.


\bibitem{villani} C. Villani, \textit{Optimal Transport: Old and New}, Springer, 2008.


\end{thebibliography}
\end{document}